\documentclass[a4paper,10pt]{article}

\usepackage[english]{babel}
\usepackage{mathtools,amsmath,amssymb,amsfonts,amsthm,mathrsfs}
\usepackage{bbm}
\usepackage{indentfirst}
\mathtoolsset{showonlyrefs}

\usepackage{graphicx,color}
\usepackage{tikz,pgfplots,pgf}
\usepackage{epsfig}
\usepackage{subcaption}
\usepackage[font=small]{caption}

\newtheorem{theorem}{Theorem}[section]

\newtheorem{lemma}{Lemma}[section]

\newtheorem{corollary}{Corollary}[section]

\theoremstyle{definition}
\newtheorem{remark}{Remark}[section]
\newtheorem*{myexample}{Prototypical Example}

\numberwithin{equation}{section}

\newcommand\blfootnote[1]{\begingroup\renewcommand\thefootnote{}\footnote{#1}\addtocounter{footnote}{-1}\endgroup}

\begin{document}

\title{
{\bf\Large Three positive solutions to an indefinite Neumann problem: a shooting method\,}\footnote{
Work partially supported by the Grup\-po Na\-zio\-na\-le per l'Anali\-si Ma\-te\-ma\-ti\-ca, la Pro\-ba\-bi\-li\-t\`{a} e le lo\-ro
Appli\-ca\-zio\-ni (GNAMPA) of the Isti\-tu\-to Na\-zio\-na\-le di Al\-ta Ma\-te\-ma\-ti\-ca (INdAM). Progetto di Ricerca 2017:
``Problemi differenziali con peso indefinito: tra metodi topologici e aspetti dinamici''. Guglielmo Feltrin is also partially supported by the project ``Existence and asymptotic behavior of solutions to systems of semilinear elliptic partial differential equations'' (T.1110.14) of the \textit{Fonds de la Recherche Fondamentale Collective}, Belgium.}
}
\author{{\bf\large Guglielmo Feltrin}
\vspace{1mm}\\
{\it\small Department of Mathematics, University of Mons}\\
{\it\small place du Parc 20, B-7000 Mons, Belgium}\\
{\it\small e-mail: guglielmo.feltrin@umons.ac.be}\vspace{1mm}\\
\vspace{1mm}\\
{\bf\large Elisa Sovrano}
\vspace{1mm}\\
{\it\small Department of Mathematics, Computer Science and Physics, University of Udine}\\
{\it\small via delle Scienze 206, 33100 Udine, Italy}\\
{\it\small e-mail: elisa.sovrano@spes.uniud.it}\vspace{1mm}}

\date{}

\maketitle

\vspace{-2mm}

\begin{center}
\normalsize \textit{Dedicated to Professor Fabio Zanolin for his coming 65th birthday}
\end{center}

\vspace{5mm}

\begin{abstract}
\noindent
We deal with the Neumann boundary value problem
\begin{equation*}
\begin{cases}
\, u'' + \bigl{(} \lambda a^{+}(t)-\mu a^{-}(t) \bigr{)}g(u) = 0, \\
\, 0 < u(t) < 1, \quad \forall\, t\in\mathopen{[}0,T\mathclose{]},\\
\, u'(0) = u'(T) = 0,
\end{cases}
\end{equation*}
where the weight term has two positive humps separated by a negative one and $g\colon \mathopen{[}0,1\mathclose{]} \to \mathbb{R}$ is a continuous function such that $g(0)=g(1)=0$, $g(s) > 0$ for $0<s<1$ and $\lim_{s\to0^{+}}g(s)/s=0$. We prove the existence of three solutions when $\lambda$ and $\mu$ are positive and  sufficiently large.
\end{abstract}

\section{Introduction}\label{section-1}

In this paper, we are interested in the multiplicity of positive solutions for an indefinite Neumann boundary value problem of the form
\begin{equation}\label{eq-1.1}
\begin{cases}
\, u'' + a(t)g(u) = 0, \\
\, u'(0) = u'(T) = 0,
\end{cases}
\end{equation}
where the weight term $a\in L^{1}(0,T)$ is sign-changing and the nonlinearity $g\colon \mathopen{[}0,1\mathclose{]} \to \mathbb{R}^{+}:=\mathopen{[}0,+\infty\mathclose{[}$ is a continuous function such that
\begin{equation*}
g(0) = g(1) = 0, \qquad g(s) > 0 \quad \text{for } \, 0 < s < 1,
\leqno{(g_{*})}
\end{equation*}
and
\begin{equation*}
\lim_{s\to 0^{+}} \dfrac{g(s)}{s} = 0.
\leqno{(g_{0})}
\end{equation*}
In our context, a \textit{solution} $u(t)$ of problem \eqref{eq-1.1} is meant in the Carath\'{e}odory's sense and is such that $0\leq u(t)\leq1$ for all $t\in\mathopen{[}0,T\mathclose{]}$. Moreover, we say that a solution is \textit{positive} if $u(t)>0$ for all $t\in\mathopen{[}0,T\mathclose{]}$.

Our study is motivated by the results achieved in \cite{BoGoHa-05,Bo-11,BoFeZa-17tams,FeZa-15ade,FeZa-15jde,FeZa-17,GaHaZa-03,GoLo-00,Lo-00} in which, dealing with different boundary value problems compared to the one treated here, the authors established multiplicity results of positive solutions in relation to the features of the graph of the weight $a(t)$.
This way, we would like to pursue further the investigation of the dynamical effects produced by the weight term associated with nonlinearities satisfying conditions $(g_{*})$ and $(g_{0})$. 
With this purpose, first of all we notice that the search of positive solutions under these assumptions leads to some well know facts. Indeed, it is straightforward to verify that problem \eqref{eq-1.1} has two trivial solutions: $u\equiv 0$ and $u\equiv 1$. Furthermore, by an integration of the differential equation in \eqref{eq-1.1} and by taking into account the Neumann boundary conditions, we obtain a necessary condition for the existence of nontrivial positive solutions to problem~\eqref{eq-1.1}: the weight $a(t)$ has to be sign-changing in the interval $\mathopen{[}0,T\mathclose{]}$. This is the peculiar characteristic which leads to call \textit{indefinite} the problem we are studying.

Indefinite Neumann problems with a nonlinearity $g(s)$ satisfying $(g_{*})$ are a very important issue in the field of population genetics, starting from the pioneering works \cite{BrHe-90,Fl-75,He-81,Se-83}. 
However, as far as we know, in order to achieve both results of uniqueness and multiplicity, lot of attention has been given to the proprieties of the nonlinearity $g(s)$ instead of the shape of the graph of the weight $a(t)$. Contributions in this direction are achieved in \cite{LoNa-02,LoNaNi-13,LoNiSu-10,So-17}. In particular, dealing with a nonlinearity $g(s)$ similar to the one taken into account in the present paper, in \cite{LoNiSu-10} the authors stated the following multiplicity result: if $\int_{0}^{T}a(t)\,dt<0$ and $g(s)$ satisfies $(g_{*})$ and $(g_{0})$ along with $\lim_{s\to0^{+}}g(s)/s^{k}>0$ for some $k>1$, then the Neumann problem associated with $d u''+a(t)g(u)=0$ has at least two positive solutions for $d>0$ sufficiently small. Instead, in the present work, we consider a different dispersal parameter $d$ and we study the effects that an indefinite weight has on the dynamics of problem~\eqref{eq-1.1}. For this reason, we suppose that the function $a(t)$ is characterized by a finite sequence of positive and negative humps. This way, our first goal is to show how the dynamics could be more rich than the ones in \cite{LoNiSu-10} (cf.~Example at p.~\pageref{example}).

Accordingly, throughout this paper, we will assume that
\begin{equation*}
\begin{aligned}
&\text{\,$\exists\,\sigma,\tau$ with $0 < \sigma < \tau < T$ such that}\\
&\begin{aligned}
&a^{+}(t)\succ 0,  & &a^{-}(t)\equiv 0, & &\text{on } \mathopen{[}0,\sigma \mathclose{]},\\
&a^{+}(t)\equiv 0, & &a^{-}(t)\succ 0,  & &\text{on } \mathopen{[}\sigma,\tau \mathclose{]}, \\
&a^{+}(t)\succ 0,  & &a^{-}(t)\equiv 0, & &\text{on } \mathopen{[}\tau,T \mathclose{]},
\end{aligned}
\end{aligned}
\leqno{(a_{*})}
\end{equation*}
where, following a standard notation, $a(t) \succ 0$ means that $a(t)\geq 0$ almost everywhere 
on a given interval with $a\not\equiv 0$ on that interval.
Moreover, given two real positive parameters $\lambda$ and $\mu$, we will consider the function 
\begin{equation}\label{eq-weight}
a(t)=a_{\lambda,\mu}(t):=\lambda a^{+}(t)-\mu a^{-}(t),
\end{equation}
with $a^{+}(t)$ and $a^{-}(t)$ denoting the positive and the negative part of the function $a(t)$, respectively.
In our framework, the dispersal parameter is thus modulated by the coefficients $\lambda$ and $\mu$.
A weight term defined as in \eqref{eq-weight} is already addressed in different contexts (cf.~\cite{BoFeZa-17tams,SoZa-17}) and the starting interest can be traced back at the works \cite{Lo-97,Lo-00}.

With the above notation, problem~\eqref{eq-1.1} reads as follows
\begin{equation*}
\begin{cases}
\, u'' + \bigl{(}\lambda a^{+}(t)-\mu a^{-}(t)\bigr{)} g(u) = 0, \\
\, u'(0) = u'(T) = 0.
\end{cases}
\leqno{(\mathcal{N}_{\lambda,\mu})}
\end{equation*}

We are now ready to state our main result which addresses multiplicity of positive solutions to problem~$(\mathcal{N}_{\lambda,\mu})$.

\begin{theorem}\label{th-1.1}
Let $g \colon \mathopen{[}0,1\mathclose{]} \to \mathbb{R}^{+}$ be a locally Lipschitz continuous function satisfying $(g_{*})$ and $(g_{0})$.
Let $a \colon \mathopen{[}0,T\mathclose{]} \to \mathbb{R}$ be an $L^{1}$-function satisfying $(a_{*})$. Then, there exists $\lambda^{*}>0$ such that for each $\lambda>\lambda^{*}$ there exists $\mu^{*}(\lambda)>0$ such that for every $\mu>\mu^{*}(\lambda)$ problem $(\mathcal{N}_{\lambda,\mu})$ has at least three positive solutions.
\end{theorem}

To illustrate the dynamics of the parameter-dependent problem $(\mathcal{N}_{\lambda,\mu})$, we will look at the following example.

\begin{myexample}\label{example}
Consider the nonlinearity defined as
\begin{equation}\label{eq-1.2}
g(s):=s^{2}(1-s), \quad s\in\mathopen{[}0,1\mathclose{]},
\end{equation} 
coupled with a weight term of the form
\begin{equation}\label{eq-1.3}
a(t):=a_{1}\mathbbm{1}_{\mathopen{[}0,\sigma\mathclose{]}}(t)-a_{2}\mathbbm{1}_{\mathopen{]}\sigma,\tau\mathclose{[}}(t)+a_{3}\mathbbm{1}_{\mathopen{[}\tau,T\mathclose{]}}(t), \quad t\in\mathopen{[}0,T\mathclose{]},
\end{equation}
where $a_{1},a_{2},a_{3}\in\mathopen{]}0,+\infty\mathclose{[}$ are some fixed values and 
$\mathbbm{1}_{A}$ denotes the indicator function of set $A$ (see Figure~\ref{fig-1}~(a)--(b) for a graphical representation of these functions). The resulting problem $(\mathcal{N}_{\lambda,\mu})$ is in the setting of Theorem~\ref{th-1.1} (see Figure~\ref{fig-1}~(c) for the numerical evidence of the existence of three positive solutions to problem $(\mathcal{N}_{\lambda,\mu})$ for $\lambda$ and $\mu$ sufficiently large).
$\hfill\lhd$
\end{myexample}

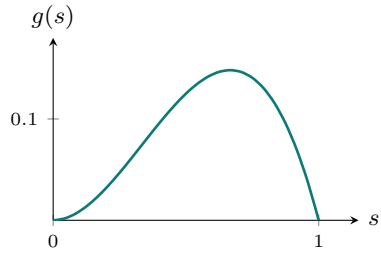
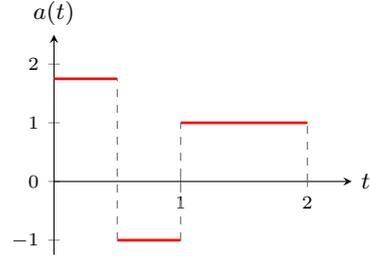
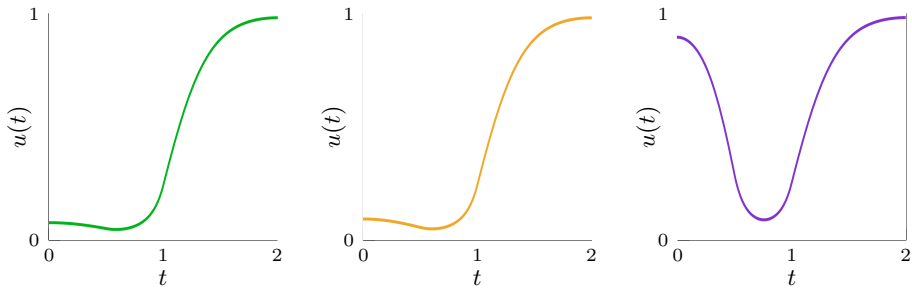
\begin{figure}[htb]
\centering
\begin{subfigure}[t]{0.45\textwidth}
\centering
\begin{tikzpicture}[scale=1]
\begin{axis}[
  tick label style={font=\scriptsize},
  axis y line=left, 
  axis x line=bottom,
  xtick={0,1},
  ytick={0.1},
  xlabel={\small $s$},
  ylabel={\small $g(s)$},
every axis x label/.style={
    at={(ticklabel* cs:1.0)},
    anchor=west,
},
every axis y label/.style={
    at={(ticklabel* cs:1.0)},
    anchor=south,
},
  width=5.6cm, 
  height=4cm,
  xmin=0,
  xmax=1.15,
  ymin=0,
  ymax=0.18] 
\addplot [mark=none,domain=0:1,line width=1pt,smooth,color={rgb:red,0.1;green,0.9;blue,0.9}] {x*x*(1-x)};
\end{axis}
\end{tikzpicture}
\caption{Graph of the nonlinear term $g(s)$ defined as in \eqref{eq-1.2}.} 
\end{subfigure}
\hspace*{\fill}
\begin{subfigure}[t]{0.45\textwidth}
\centering
\begin{tikzpicture}[scale=1]
\begin{axis}[
  tick label style={font=\scriptsize},
  axis y line=left, 
  axis x line=middle,
  xtick={0,1,2},
  ytick={-1,0,1,2},
  xlabel={\small $t$},
  ylabel={\small $a(t)$},
every axis x label/.style={
    at={(ticklabel* cs:1.0)},
    anchor=west,
},
every axis y label/.style={
    at={(ticklabel* cs:1.0)},
    anchor=south,
},
  width=5.5cm,
  height=4.5cm,
  xmin=0,
  xmax=2.35,
  ymin=-1.25,
  ymax=2.5] 
\addplot [mark=none,domain=0:0.5,line width=1pt,smooth,red] {1.75};
\addplot [mark=none,domain=0.5:1,line width=1pt,smooth,red] {-1};
\addplot [mark=none,domain=1:2,line width=1pt,smooth,red] {1};
\addplot[mark=none,domain=0:2,line width=0.5pt,dashed,gray] coordinates {(0.5,-1)(0.5,1.75)};
\addplot[mark=none,domain=0:2,line width=0.5pt,dashed,gray] coordinates {(1,-1)(1,1)};
\addplot[mark=none,domain=0:2,line width=0.5pt,dashed,gray] coordinates {(2,0)(2,1)};
\end{axis}
\end{tikzpicture}
\caption{Graph of the weight term $a(t)$ defined as in \eqref{eq-1.3} with $\sigma=0.5$, $\tau=1$, $T=2$, $a_{1}=1.75$, $a_{2}=1$, $a_{3}=1$.}
\end{subfigure}
\\
\vspace{15pt}
\begin{subfigure}[t]{1\textwidth}
\centering
\begin{tikzpicture}[scale=1]
\begin{axis}[
  tick pos=left,
  tick label style={font=\scriptsize},
          scale only axis,
  enlargelimits=false,
  xtick={0,1,2},
  ytick={0,1},
  xlabel={\small $t$},
  ylabel={\small $u(t)$},
  max space between ticks=50,
                minor x tick num=4,
                minor y tick num=10,  
every axis x label/.style={
below,
at={(1.5cm,0cm)},
  yshift=-8pt
  },
every axis y label/.style={
below,
at={(0cm,1.5cm)},
  xshift=-3pt},
  y label style={rotate=90,anchor=south},
  width=3cm,
  height=3cm,  
  xmin=0,
  xmax=2,
  ymin=0,
  ymax=1] 
\addplot graphics[xmin=0,xmax=2,ymin=0,ymax=1] {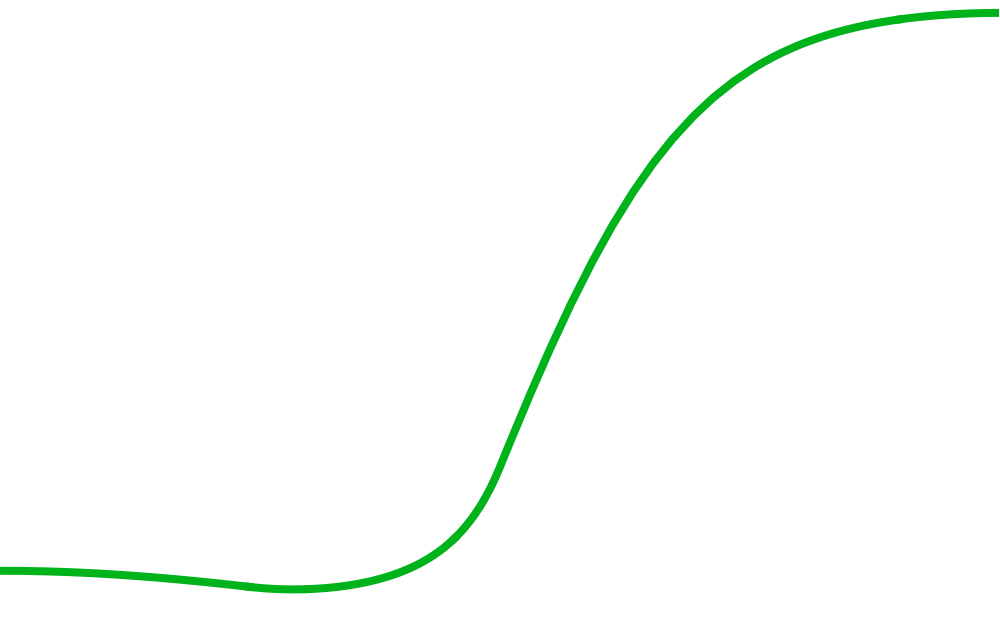};
\end{axis}
\end{tikzpicture} 
\hspace*{\fill}
\begin{tikzpicture}[scale=1]
\begin{axis}[
  tick pos=left,
  tick label style={font=\scriptsize},
          scale only axis,
  enlargelimits=false,
  xtick={0,1,2},
  ytick={0,1},
  xlabel={\small $t$},
  ylabel={\small $u(t)$},
  max space between ticks=50,
                minor x tick num=4,
                minor y tick num=10,  
  extra x tick labels={,,},
every axis x label/.style={
below,
at={(1.5cm,0cm)},
  yshift=-8pt
  },
every axis y label/.style={
below,
at={(0cm,1.5cm)},
  xshift=-3pt},
  every extra axis y label/.style={xshift=-0.5pt},
  y label style={rotate=90,anchor=south},
  width=3cm,
  height=3cm,  
  xmin=0,
  xmax=2,
  ymin=0,
  ymax=1] 
\addplot graphics[xmin=0,xmax=2,ymin=0,ymax=1] {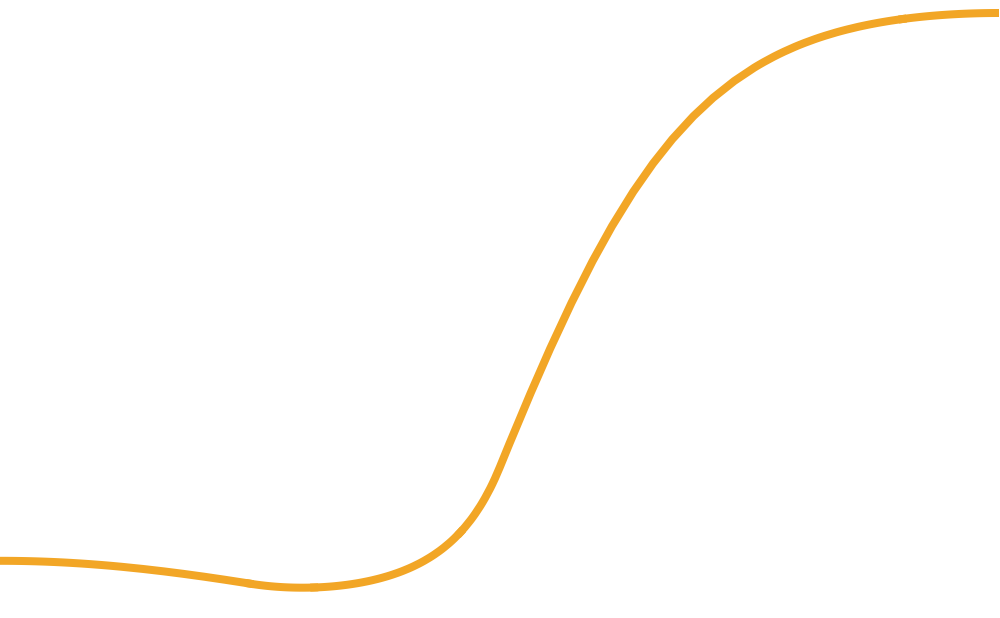};
\end{axis}
\end{tikzpicture} 
\hspace*{\fill}
\begin{tikzpicture}[scale=1]
\begin{axis}[
  tick pos=left,
  tick label style={font=\scriptsize},
          scale only axis,
  enlargelimits=false,
  xtick={0,1,2},
  ytick={0,1},
  max space between ticks=50,
                minor x tick num=4,
                minor y tick num=10,                
  xlabel={\small $t$},
  ylabel={\small $u(t)$},
every axis x label/.style={
below,
at={(1.5cm,0cm)},
  yshift=-8pt
  },
every axis y label/.style={
below,
at={(0cm,1.5cm)},
  xshift=-3pt},
  y label style={rotate=90,anchor=south},
  width=3cm,
  height=3cm,  
  xmin=0,
  xmax=2,
  ymin=0,
  ymax=1] 
\addplot graphics[xmin=0,xmax=2,ymin=0,ymax=1] {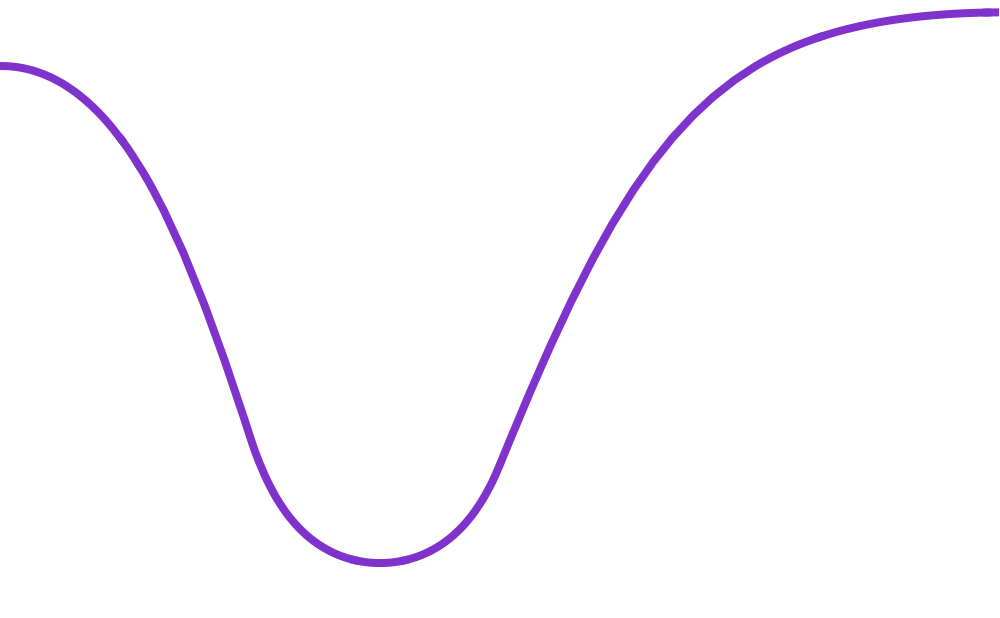};
\end{axis}
\end{tikzpicture} 
\caption{Three positive solutions of $(\mathcal{N}_{\lambda,\mu})$ for $\lambda=25$ and $\mu=500$.}
\end{subfigure} 
\caption{Multiplicity of positive solutions for the indefinite Neumann problem $(\mathcal{N}_{\lambda,\mu})$ as in the framework of Example.}        
\label{fig-1}
\end{figure}

As a strategy to prove Theorem~\ref{th-1.1}, we exploit the shooting method. With this respect, it is natural to study problem $(\mathcal{N}_{\lambda,\mu})$ in the phase-plane $(x,y)=(u,u')$. Accordingly, the differential equation in $(\mathcal{N}_{\lambda,\mu})$ can be equivalently written as a planar system in the following form
\begin{equation*}
\begin{cases}
\, x' = y, \\
\, y' = - \bigl{(} \lambda a^{+}(t) - \mu a^{-}(t) \bigr{)} g(x).
\end{cases}
\leqno{(\mathcal{S}_{\lambda,\mu})}
\end{equation*}
Thus, we consider the vector field associated with $(\mathcal{S}_{\lambda,\mu})$ in order to look at the corresponding deformation of the interval $\mathopen{[}0,1\mathclose{]}$ contained in the positive $x$-axis.
As an intuitive explanation of our approach, we state that we look for intersection points between two planar continua: the one obtained from shooting the set $X_{\mathopen{[}0,1\mathclose{]}}:=\mathopen{[}0,1\mathclose{]}\times\{0\}$ forward in time over $\mathopen{[}0,\tau\mathclose{]}$ with the other one obtained from shooting again the same set $X_{\mathopen{[}0,1\mathclose{]}}$ backward in time over $\mathopen{[}\tau,T\mathclose{]}$. We refer to Section~\ref{section-2.4} for a rigorous discussion and application of this topological tool.

\medskip

The present paper is organized as follows. In Section~\ref{section-2} we carry out the proof of Theorem~\ref{th-1.1}. In Section~\ref{section-3} we end with applications, focusing on the periodic problem and on radially symmetric solutions of elliptic PDEs in annular domains.

\section{Proof of Theorem~\ref{th-1.1}}\label{section-2}

In this section we will find three positive solutions for problem~$(\mathcal{N}_{\lambda,\mu})$, demonstrating our multiplicity result. Accordingly, the proof of Theorem~\ref{th-1.1} will be divided into four steps. First of all, we will study system~$(\mathcal{S}_{\lambda,\mu})$ separately in the three intervals: $\mathopen{[}0,\sigma\mathclose{]}$, $\mathopen{[}\sigma,\tau\mathclose{]}$ and $\mathopen{[}\tau,T\mathclose{]}$. Then, we will achieve the thesis combining the dynamics of system~$(\mathcal{S}_{\lambda,\mu})$ previously performed.

Throughout this section, we assume the following hypotheses on problem $(\mathcal{N}_{\lambda,\mu})$: the weight $a\in L^{1}(0,T)$ satisfies $(a_{*})$ and the function $g \colon \mathopen{[}0,1\mathclose{]} \to {\mathbb{R}}^{+}$ is locally Lipschitz continuous satisfying $(g_{*})$ and $(g_{0})$. 

As a first step, we extend the function $g(s)$ continuously to the whole real line, by setting
\begin{equation*}
g(s)= 0, \quad \text{for } \, s\in\mathopen{]}-\infty,0\mathclose{[}\cup\mathopen{]}1,+\infty\mathclose{[}.
\end{equation*}
The extension is still denoted by $g(s)$. In this manner, any solution of a Cauchy problem associated with $(\mathcal{S}_{\lambda,\mu})$ is globally defined on $\mathopen{[}0,T\mathclose{]}$.

Secondly, we introduce the following notation:
\begin{equation*}
A^{\pm}(t', t''):=\int_{t'}^{t''}a^{\pm}(\xi) \,d\xi, \quad t', t''\in\mathopen{[}0,T\mathclose{]} \, \text{ with } t'\leq t''.
\end{equation*}
Moreover, we set
\begin{equation*}
g_{*}(\kappa',\kappa'') := \min_{s\in\mathopen{[}\kappa',\kappa''\mathclose{]}}g(s), \quad \kappa',\kappa''\in\mathopen{[}0,1\mathclose{]} \, \text{ with } \kappa' < \kappa''.
\end{equation*}

\subsection{Study of system $(\mathcal{S}_{\lambda,\mu})$ in $\mathopen{[}0,\sigma\mathclose{]}$}\label{section-2.1}

In the interval $\mathopen{[}0,\sigma\mathclose{]}$ system~$(\mathcal{S}_{\lambda,\mu})$ reduces to
\begin{equation}\label{IVP-1}
\begin{cases}
\, x' = y, \\
\, y' = - \lambda a^{+}(t) g(x).
\end{cases}
\end{equation}

Since $A^{+}(0,0)=0$, $A^{+}(0,\sigma)>0$ and $t\mapsto A^{+}(0,t)$ is a continuous non-decreasing map on $\mathopen{[}0,\sigma\mathclose{]}$, without loss of generality, we can suppose that
\begin{equation*}
A^{+}(0,t)>0, \quad \forall \, t\in\mathopen{]}0,\sigma\mathclose{]}.
\end{equation*}
Otherwise, there exists a maximal interval $\mathopen{[}0,t_{0}\mathclose{]}$ where $A(0,t)=0$ for all $t\in\mathopen{[}0,t_{0}\mathclose{]}$ and the study of system \eqref{IVP-1} can be performed in the interval $\mathopen{[}t_{0},\sigma\mathclose{]}$.

We are going to prove that, for every initial condition $(x_{0},0)$ with $x_{0}\in\mathopen{]}0,1\mathclose{[}$, the solution $(x(t),y(t))$ of the Cauchy problem associated with \eqref{IVP-1} at time $t=\sigma$ belongs to $\mathopen{]}-\infty,0\mathclose{]}\times\mathopen{]}-\infty,0\mathclose{[}$ for $\lambda$ sufficiently large.

\begin{lemma}\label{lem-2.1}
Let $\lambda>0$, $\kappa_{1}\in\mathopen{]}0,1\mathclose{[}$ and $t_{1}\in\mathopen{]}0,\sigma\mathclose{[}$. For every $\gamma_{1} \geq \kappa_{1}/(\sigma-t_{1})$, any solution $(x(t),y(t))$ of \eqref{IVP-1} with $x(t_{1}) \leq \kappa_{1}$ and $y(t_{1}) \leq -\gamma_{1}$ satisfies $x(\sigma) \leq 0$ and $y(\sigma)\leq -\gamma_{1}$.
\end{lemma}

\begin{proof}
Let $\lambda, \kappa_{1},t_{1}$ and $\gamma_{1}$ be fixed as in the statement.
Let $(x(t),y(t))$ be a solution of \eqref{IVP-1} with $x(t_{1}) \leq \kappa_{1}$ and $y(t_{1}) \leq -\gamma_{1}$.
Since $y'(t)\leq0$ on $\mathopen{[}0,\sigma\mathclose{]}$, we immediately obtain that
\begin{equation*}
y(t) \leq y(t_{1}) \leq -\gamma_{1}, \quad \text{for all } \, t\in \mathopen{[}t_{1},\sigma\mathclose{]},
\end{equation*}
and, consequently, we have
\begin{equation*}
x(\sigma) = x(t_{1}) + \int_{t_{1}}^{\sigma} y(\xi) \,d\xi \leq \kappa_{1} - \gamma_{1}(\sigma-t_{1}) \leq 0.
\end{equation*}
The thesis follows.
\end{proof}

\begin{lemma}\label{lem-2.2}
Let $\kappa_{0},\kappa_{1}$ be such that $0 < \kappa_{1} < \kappa_{0} < 1$ and $t_{1} \in\mathopen{]}0,\sigma\mathclose{[}$.
Given 
\begin{equation}\label{eq-lambda1}
\lambda^{\star}(\kappa_{0},\kappa_{1},t_{1}) := \dfrac{\kappa_{0}-\kappa_{1}}{g_{*}(\kappa_{1},\kappa_{0})\int_{0}^{t_{1}} A^{+}(0,\xi) \,d\xi}
\end{equation}
and $0<\gamma_{1} \leq (\kappa_{0}-\kappa_{1})/t_{1}$, then, for every $\lambda > \lambda^{\star}(\kappa_{0},\kappa_{1},t_{1})$, the solution $(x(t),y(t))$ of \eqref{IVP-1} with initial conditions $x(0) = \kappa_{0}$ and $y(0)=0$ satisfies $x(t_{1}) < \kappa_{1}$ and $y(t_{1}) < -\gamma_{1}$.
\end{lemma}

\begin{proof}
Let $\kappa_{0},\kappa_{1},t_{1},\gamma_{1}$ and $\lambda^{\star}(\kappa_{0},\kappa_{1},t_{1})$ be fixed as in the statement. For $\lambda > \lambda^{\star}(\kappa_{0},\kappa_{1},t_{1})$ consider the solution $(x(t),y(t))$ of \eqref{IVP-1} with $x(0) = \kappa_{0}$ and $y(0)=0$. 

First, we suppose by contradiction that $x(t_{1}) \geq \kappa_{1}$. Consequently, by the monotonicity of $x(t)$ in $\mathopen{[}0,\sigma\mathclose{]}$, we have
\begin{equation*}
0 < \kappa_{1} \leq x(t) \leq \kappa_{0} < 1, \quad \text{for all } \, t\in\mathopen{[}0,t_{1}\mathclose{]}.
\end{equation*}
Since $y'(t) \leq -\lambda a^{+}(t)g_{*}(\kappa_{1},\kappa_{0})$ on $\mathopen{[}0,t_{1}\mathclose{]}$, we obtain
\begin{equation*}
y(t) \leq -\lambda g_{*}(\kappa_{1},\kappa_{0}) A^{+}(0,t), \quad \text{for all } \, t\in\mathopen{[}0,t_{1}\mathclose{]}.
\end{equation*}
Then
\begin{equation*}
x(t) \leq x(0)-\lambda g_{*}(\kappa_{1},\kappa_{0}) \int_{0}^{t}A^{+}(0,\xi)\,d\xi, \quad \text{for all } \, t\in\mathopen{[}0,t_{1}\mathclose{]},
\end{equation*}
and, since $\lambda > \lambda^{\star}(\kappa_{0},\kappa_{1},t_{1})$, in particular we have
\begin{equation*}
x(t_{1}) \leq \kappa_{0}-\lambda g_{*}(\kappa_{1},\kappa_{0}) \int_{0}^{t_{1}}A^{+}(0,\xi)\,d\xi < \kappa_{1},
\end{equation*}
a contradiction.

Secondly, we suppose by contradiction that
\begin{equation*}
y(t) \geq -\gamma_{1}, \quad \text{for all } \, t\in\mathopen{[}0,t_{1}\mathclose{]}.
\end{equation*}
By integrating, we have
\begin{equation*}
x(t_{1}) = \kappa_{0} + \int_{0}^{t_{1}} y(\xi)\,d\xi \geq \kappa_{0}-\gamma_{1} t_{1} \geq \kappa_{1}.
\end{equation*}
A contradiction is achieved as above and the lemma is proved.
\end{proof}

Notice that in the previous lemmas condition $(g_{0})$ is not required to obtain the conclusions. On the contrary, in the next lemma this condition will be crucial (cf.~Remark~\ref{rem-2.2}). In more detail, for any fixed $\lambda>0$, taking initial conditions $(x(0),y(0))\in\mathopen{]}0,\delta\mathclose{]}\times\{0\}$ with $\delta>0$ small, then the solution $(x(t),y(t))$ of the Cauchy problem associated with system~\eqref{IVP-1} at time $t=\sigma$ belongs to $\mathopen{]}0,1\mathclose{[}\times\mathopen{]}-\infty,0\mathclose{[}$.

\begin{lemma}\label{lem-ang}
Let $\lambda>0$, $\nu\in\mathopen{]}0,\pi/2\mathclose{[}$ and $\kappa_{1}\in\mathopen{]}0,1\mathclose{[}$. Then, there exists $\hat{\varepsilon}=\hat{\varepsilon}(\lambda,\nu)>0$ such that for any $\varepsilon\in\mathopen{]}0,\hat{\varepsilon}\mathclose{[}$ there exists $\delta_{\varepsilon}\in\mathopen{]}0,\kappa_{1}\mathclose{[}$ such that the following holds: for any fixed $\kappa\in\mathopen{]}0,\delta_{\varepsilon}\mathclose{]}$, the solution $(x(t),y(t))$ of \eqref{IVP-1} with initial conditions $x(0)=\kappa$ and $y(0)=0$ satisfies
\begin{equation}\label{eq-arctan}
-\nu \leq \arctan \biggl{(} \dfrac{y(t)}{x(t)} \biggr{)} \leq 0, \quad \text{for all } \, t\in \mathopen{[}0,\sigma\mathclose{]}.
\end{equation}
\end{lemma}

\begin{proof}
Let $\lambda, \nu$ and $\kappa_{1}$ be fixed as in the statement.
Let $\hat{\varepsilon}=\hat{\varepsilon}(\lambda,\nu)>0$ be such that
\begin{equation}\label{eq-epsilon}
\arctan\Bigl{(}\sqrt{\lambda\|a^{+}\|_{\infty}\varepsilon} \tan\bigr{(}\sigma\sqrt{\lambda\|a^{+}\|_{\infty}\varepsilon} \bigr{)}\Bigr{)} < \nu, \quad \text{for all } \, \varepsilon\in\mathopen{]}0,\hat{\varepsilon}\mathclose{[},
\end{equation}
where, as usual, we denote by $\|\cdot\|_{\infty}$ the supremum norm.
From hypothesis $(g_{0})$, for all $\varepsilon>0$ there exists $\delta_{\varepsilon}\in\mathopen{]}0,\kappa_{1}\mathclose{[}$ such that
\begin{equation*}
g(s) \leq \varepsilon s, \quad \text{for all } \, s\in\mathopen{[}0,\delta_{\varepsilon}\mathclose{]}.
\end{equation*}
For $\kappa\in\mathopen{]}0,\delta_{\varepsilon}\mathclose{]}$, we consider the solution $(x(t),y(t))$ of \eqref{IVP-1} with $x(0)=\kappa$ and $y(0)=0$.

First of all, we write the solution in polar coordinates
\begin{equation*}
x(t) = \rho(t) \cos(\vartheta(t)), \quad y(t) = \rho(t) \sin(\vartheta(t)).
\end{equation*}
For $t\in \mathopen{[}0,\sigma\mathclose{]}$ we have
\begin{equation*}
\vartheta(t) = \arctan \biggl{(} \dfrac{y(t)}{x(t)} \biggr{)}
\end{equation*}
and
\begin{equation*}
\vartheta'(t) = \dfrac{y'(t)x(t)-x'(t)y(t)}{x^{2}(t)+y^{2}(t)} = \dfrac{-\lambda a^{+}(t) g(x(t)) x(t) - y^{2}(t)}{\rho^{2}(t)} \leq 0.
\end{equation*}
Then, since $\vartheta(0)=0$,
\begin{equation*}
\vartheta(t) = \arctan \biggl{(} \dfrac{y(t)}{x(t)} \biggr{)} \leq 0, \quad \text{for all } \, t\in \mathopen{[}0,\sigma\mathclose{]}.
\end{equation*}
Therefore, in order to conclude the proof, we have to prove the first inequality in \eqref{eq-arctan}.
Let $\varepsilon\in\mathopen{]}0,\hat{\varepsilon}\mathclose{[}$, then
\begin{align*}
-\vartheta'(t) 
&= \dfrac{\lambda a^{+}(t) g(x(t)) x(t) + y^{2}(t)}{\rho^{2}(t)}
\leq \dfrac{\lambda a^{+}(t) \varepsilon x^{2}(t) + y^{2}(t)}{\rho^{2}(t)}
\\ &\leq \lambda \|a^{+}\|_{\infty} \varepsilon \cos^{2}(\vartheta(t)) + \sin^{2}(\vartheta(t)),\quad \text{for all } \, t\in \mathopen{[}0,\sigma\mathclose{]}.
\end{align*}
By integrating on $\mathopen{[}0,t\mathclose{]}$, we obtain
\begin{equation*}
-\int_{\vartheta(0)}^{\vartheta(t)} \dfrac{d\zeta}{\lambda\|a^{+}\|_{\infty}\varepsilon \cos^{2}(\zeta) + \sin^{2}(\zeta) }
\leq \int_{0}^{t} d\xi = t \leq \sigma, \quad \text{for all } \, t\in \mathopen{[}0,\sigma\mathclose{]}.
\end{equation*}
The first term can be equivalently written as
\begin{align*}
&-\int_{\vartheta(0)}^{\vartheta(t)} \dfrac{d\zeta}{\lambda\|a^{+}\|_{\infty}\varepsilon \cos^{2}(\zeta) + \sin^{2}(\zeta) } =
\\ &= \int_{\vartheta(t)}^{0} \dfrac{d\zeta}{\cos^{2}(\zeta) \bigr{(}\lambda\|a^{+}\|_{\infty}\varepsilon + \tan^{2}(\zeta) \bigl{)} }
\\ &= -\int_{\tan(\vartheta(t))}^{0} \dfrac{dz}{\lambda\|a^{+}\|_{\infty}\varepsilon + z^{2}}
\\ &= \dfrac{1}{\sqrt{\lambda\|a^{+}\|_{\infty}\varepsilon}}\arctan\biggl{(}\dfrac{\tan|\vartheta(t)|}{\sqrt{\lambda\|a^{+}\|_{\infty}\varepsilon}} \biggr{)}, \quad \text{for all } \, t\in \mathopen{[}0,\sigma\mathclose{]}.
\end{align*}
Consequently
\begin{equation*}
|\vartheta(t)| \leq \arctan\Bigl{(}\sqrt{\lambda\|a^{+}\|_{\infty}\varepsilon} \tan\bigr{(}\sigma\sqrt{\lambda\|a^{+}\|_{\infty}\varepsilon} \bigr{)}\Bigr{)}, \quad \text{for all } \, t\in \mathopen{[}0,\sigma\mathclose{]}.
\end{equation*}
At last, by the choice of $\varepsilon\in\mathopen{]}0,\hat{\varepsilon}\mathclose{[}$ and condition \eqref{eq-epsilon}, formula \eqref{eq-arctan} is proved.
\end{proof}

\subsection{Study of system $(\mathcal{S}_{\lambda,\mu})$ in $\mathopen{[}\tau,T\mathclose{]}$}\label{section-2.2}

System~$(\mathcal{S}_{\lambda,\mu})$ in the interval $\mathopen{[}\tau,T\mathclose{]}$ can be equivalently written as \eqref{IVP-1}.
Since $A^{+}(T,T)=0$, $A^{+}(\tau,T)>0$ and $t\mapsto A^{+}(t,T)$ is a continuous non-increasing map on $\mathopen{[}\tau,T\mathclose{]}$, without loss of generality, we can suppose that
\begin{equation*}
A^{+}(t,T)>0, \quad \forall \, t\in\mathopen{[}\tau,T\mathclose{[}.
\end{equation*}
Otherwise, there exists a maximal interval $\mathopen{[}t_{T},T\mathclose{]}$ where $A^{+}(t,T)=0$ for all $t\in\mathopen{[}t_{T},T\mathclose{]}$ and the study of system \eqref{IVP-1} can be performed in the interval $\mathopen{[}\tau,t_{T}\mathclose{]}$.

In this context, the situation is exactly symmetric to the one described in Lemma~\ref{lem-2.1} and Lemma~\ref{lem-2.2}. We collect here the corresponding results, passing over the proofs since they are analogous to the previous ones.

\begin{lemma}\label{lem-2.4}
Let $\lambda>0$, $\kappa_{3}\in\mathopen{]}0,1\mathclose{[}$ and $t_{3}\in\mathopen{]}\tau,T\mathclose{[}$. 
For every $\gamma_{3} \geq \kappa_{3}/(t_{3}-\tau)$, any solution $(x(t),y(t))$ of \eqref{IVP-1} with $x(t_{3}) \leq \kappa_{3}$ and $y(t_{3}) \geq \gamma_{3}$ satisfies $x(\tau) \leq 0$ and $y(\tau)\geq \gamma_{3}$.
\end{lemma}

\begin{lemma}\label{lem-2.5}
Let $\kappa_{3},\kappa_{T}$ be such that $0 < \kappa_{3} < \kappa_{T} < 1$ and $t_{3}\in\mathopen{]}\tau,T\mathclose{[}$. 
Given 
\begin{equation}\label{eq-lambda2}
\lambda^{\star\star}(\kappa_{3},\kappa_{T},t_{3}) := \dfrac{\kappa_{T}-\kappa_{3}}{g_{*}(\kappa_{3},\kappa_{T})\int_{t_{3}}^{T} A^{+}(\xi,T) \,d\xi}
\end{equation}
and $0<\gamma_{3}\leq(\kappa_{T}-\kappa_{3})/(T-t_{3})$, then, for every $\lambda > \lambda^{\star\star}(\kappa_{3},\kappa_{T},t_{3})$, the solution $(x(t),y(t))$ of \eqref{IVP-1} with initial conditions $x(T) = \kappa_{T}$ and $y(T)=0$ satisfies $x(t_{3}) < \kappa_{3}$ and $y(t_{3}) > \gamma_{3}$.
\end{lemma}

\subsection{Study of system $(\mathcal{S}_{\lambda,\mu})$ in $\mathopen{[}\sigma,\tau\mathclose{]}$}\label{section-2.3}

Consider now the interval $\mathopen{[}\sigma,\tau\mathclose{]}$, where system~$(\mathcal{S}_{\lambda,\mu})$ reduces to
\begin{equation}\label{IVP-2}
\begin{cases}
\, x' = y, \\
\, y' = \mu a^{-}(t) g(x).
\end{cases}
\end{equation}

Without loss of generality, we can suppose that $A^{-}(\sigma,t)>0$ for all $t\in\mathopen{]}\sigma,\tau\mathclose{]}$. Indeed, it is always possible to choose a suitable $\sigma$ as in $(a_{*})$ that satisfies this additional hypothesis, as pointed out in \cite{BoFeZa-17tams,FeZa-15jde,FeZa-17,SoZa-17}.

Our purpose is to determine the initial conditions $(x(\sigma),y(\sigma))$ such that the corresponding solution $(x(t),y(t))$ of the Cauchy problem associated with system~\eqref{IVP-2} belongs to $\mathopen{[}1,+\infty\mathclose{[}\times\mathopen{]}0,+\infty\mathclose{[}$ at time $t=\tau$, for $\mu$ sufficiently large.

\begin{lemma}\label{lem-2.6}
Let $\mu>0$, $\kappa_{2}\in\mathopen{]}0,1\mathclose{[}$ and $t_{2}\in\mathopen{]}\sigma,\tau\mathclose{[}$. For every $\omega\geq(1-\kappa_{2})/(\tau-t_{2})$, any solution $(x(t),y(t))$ of \eqref{IVP-2} with $x(t_{2})\geq \kappa_{2}$ and $y(t_{2})\geq \omega$ satisfies $x(\tau) \geq 1$ and $y(\tau)\geq \omega$.
\end{lemma}

\begin{proof}
Let $\mu,\kappa_{2},t_{2}$ and $\omega$ be fixed as in the statement.
Let $(x(t),y(t))$ be a solution of \eqref{IVP-2} with $x(t_{2}) \geq \kappa_{2}$ and $y(t_{2}) \geq \omega$.
Since $y'(t)\geq 0$ on $\mathopen{[}\sigma,\tau\mathclose{]}$, we immediately obtain that $y(t)\geq y(t_{2}) \geq \omega$ for every $t\in \mathopen{[}t_{2},\tau\mathclose{]}$. In particular, it follows that $y(\tau)\geq \omega$. Moreover, we have 
\begin{equation*}
x(\tau) = x(t_{2}) + \int_{t_{2}}^{\tau} y(\xi) \,d\xi \geq \kappa_{2} + \omega (\tau - t_{2}) \geq 1.
\end{equation*}
The thesis follows.
\end{proof}

\begin{lemma}\label{lem-2.7}
Let $\kappa_{\sigma},\kappa_{2}$ be such that $0 < \kappa_{\sigma} < \kappa_{2} < 1$ and $\omega_{\sigma} > 0$.
Given 
\begin{equation*}
\sigma< t_{2} \leq \min\biggl{\{}\sigma + \dfrac{\kappa_{\sigma}}{2\omega_{\sigma}},\tau \biggr{\}}, \quad 0 < \omega \leq \dfrac{\kappa_{2}-\kappa_{\sigma}}{t_{2}-\sigma},
\end{equation*} 
and
\begin{equation}\label{eq-mu}
\mu^{\star}(\kappa_{2},\kappa_{\sigma},t_{2},\omega_{\sigma}) := \dfrac{\kappa_{2} - \kappa_{\sigma} + (t_{2}-\sigma) \omega_{\sigma}}{g_{*}(\kappa_{\sigma}/2,\kappa_{2}) \int_{\sigma}^{t_{2}} A^{-}(\sigma,\xi) \,d\xi},
\end{equation}
then, for every $\mu>\mu^{\star}(\kappa_{2},\kappa_{\sigma},t_{2},\omega_{\sigma})$, any solution $(x(t),y(t))$ of \eqref{IVP-2} with $x(\sigma) = \kappa_{\sigma}$ and $y(\sigma) \geq -\omega_{\sigma}$ satisfies $x(t_{2})> \kappa_{2}$ and $y(t_{2})> \omega$.
\end{lemma}

\begin{proof}
Let $\kappa_{\sigma},\kappa_{2},\omega_{\sigma},t_{2},\omega$ and $\mu^{\star}(\kappa_{2},\kappa_{\sigma},t_{2},\omega_{\sigma})$ be fixed as in the statement.
For $\mu>\mu^{\star}(\kappa_{2},\kappa_{\sigma},t_{2},\omega_{\sigma})$, let $(x(t),y(t))$ be a solution of \eqref{IVP-2} with $x(\sigma) = \kappa_{\sigma}$ and $y(\sigma) \geq -\omega_{\sigma}$.

First, we suppose by contradiction that $x(t_{2})\leq\kappa_{2}$. This way, by the convexity of the function $x(t)$ in $\mathopen{[}\sigma,\tau\mathclose{]}$ and the assumption $\kappa_{2}>\kappa_{\sigma}$, we easily deduce that 
\begin{equation*}
x(t) \leq \kappa_{2}, \quad \text{for all } \, t\in \mathopen{[}\sigma,t_{2}\mathclose{]}.
\end{equation*}
Since $y'(t)\geq0$ on $\mathopen{[}\sigma,\tau\mathclose{]}$ and $y(\sigma) \geq -\omega_{\sigma}$, we derive that
\begin{equation*}
x(t) \geq -\omega_{\sigma} t + \kappa_{\sigma} + \omega_{\sigma} \sigma,
\quad \text{for all } \, t\in\mathopen{[}\sigma,\tau\mathclose{]},
\end{equation*}
and, by the condition on the point $t_{2}$, we obtain that
\begin{equation*}
x(t) \geq \dfrac{\kappa_{\sigma}}{2}, \quad \text{for all } \, t\in \mathopen{[}\sigma,t_{2}\mathclose{]}.
\end{equation*}
By an integration of \eqref{IVP-2}, for every $t\in \mathopen{[}\sigma,t_{2}\mathclose{]}$, we have
\begin{equation*}
y(t) = y(\sigma) + \int_{\sigma}^{t} y'(\xi) \,d\xi = y(\sigma) + \mu \int_{\sigma}^{t} a^{-}(\xi) g(x(\xi)) \,d\xi
\end{equation*}
and
\begin{equation*}
x(t) = x(\sigma) + \int_{\sigma}^{t} y(\xi) \,d\xi 
= \kappa_{\sigma} + (t-\sigma) y(\sigma) + \mu \int_{\sigma}^{t} \int_{\sigma}^{z} a^{-}(\xi) g(x(\xi)) \,d\xi dz.
\end{equation*}
Then, by the choice of $\mu>\mu^{\star}(\kappa_{2},\kappa_{\sigma},t_{2},\omega_{\sigma})$, it follows that
\begin{align*}
\kappa_{2} \geq x(t_{2}) \geq \kappa_{\sigma} - (t_{2}-\sigma) \omega_{\sigma} + \mu g_{*}(\kappa_{\sigma}/2,\kappa_{2}) \int_{\sigma}^{t_{2}} A^{-}(\sigma,\xi) \,d\xi > \kappa_{2},
\end{align*}
a contradiction.

Secondly, we suppose by contradiction that $y(t_{2}) \leq \omega$ and thus that $y(t) \leq \omega$ for all $t\in \mathopen{[}\sigma,t_{2}\mathclose{]}$. 
Then
\begin{equation*}
x(t_{2}) \leq \kappa_{\sigma} + \omega (t_{2}-\sigma) \leq \kappa_{2}
\end{equation*}
and a contradiction is achieved as above. This concludes the proof.
\end{proof}

\subsection{Application of the shooting method}\label{section-2.4}

The working  hypotheses assumed in this paper ensure 
the uniqueness and the global existence of the solution $(x(t;\alpha,x_{\alpha},y_{\alpha}), y(t;\alpha, x_{\alpha},y_{\alpha}))$ 
to system~$(\mathcal{S}_{\lambda,\mu})$ satisfying the initial conditions 
\begin{equation}\label{eq-ic}
x(\alpha)=x_{\alpha},\quad y(\alpha)=y_{\alpha}.
\end{equation}
Consequently, we introduce (for every fixed couple of parameters $\lambda$ and $\mu$) the \textit{Poincar\'{e} map} $\Phi_{\alpha}^{\beta}$ associated to $(\mathcal{S}_{\lambda,\mu})$ in the interval $\mathopen{[}\alpha,\beta\mathclose{]}\subseteq\mathopen{[}0,T\mathclose{]}$.
In particular, it is a global diffeomorphism of the plane onto itself 
defined by 
\begin{equation*}
\Phi_{\alpha}^{\beta}\colon \mathbb{R}^{2} \to \mathbb{R}^{2},\quad \Phi_{\alpha}^{\beta}(x_{\alpha},y_{\alpha}) := (x(\beta), y(\beta)),
\end{equation*}
where $(x(t), y(t))=(x(t;\alpha,x_{\alpha},y_{\alpha}), y(t;\alpha,x_{\alpha},y_{\alpha}))$ is the solution to~$(\mathcal{S}_{\lambda,\mu})$ satisfying the initial conditions~\eqref{eq-ic}.

At this point our goal is to combine the results obtained in the previous subsections in order to describe the deformation in the phase-plane $(x,y)$ of the interval $X_{\mathopen{[}0,1\mathclose{]}}:=\mathopen{[}0,1\mathclose{]}\times\{0\}$ through the Poincar\'{e} map. 
In particular, it is straightforward to verify that any point 
$P\in\Phi_{0}^{\tau}(X_{\mathopen{[}0,1\mathclose{]}})\cap\Phi_{T}^{\tau}(X_{\mathopen{[}0,1\mathclose{]}})$ 
determines univocally a solution $(x(t;\tau,P),y(t;\tau,P))$ of system $(\mathcal{S}_{\lambda,\mu})$ satisfying the Neumann boundary conditions $y(0;\tau,P)=y(T;\tau,P)=0$. Hence, $u(t):=x(t;\tau,P)$ is a solution of problem $(\mathcal{N}_{\lambda,\mu})$.

In Figure~\ref{fig-2} we illustrate this approach by means of numerical simulations in the case of the leading Example considered in the present paper. 

\begin{figure}[htb]
\centering
\begin{subfigure}[t]{.5\textwidth}
  \centering
\begin{tikzpicture}[scale=1]
\begin{axis}[  
  scale only axis,
  x post scale=1,
  y post scale=0.25,
  tick label style={font=\scriptsize},
  enlargelimits=false, axis on top, axis equal image,
  xtick={-1,0,...,3},
  ytick={-2,0,...,8},
  xlabel={\small $x$},
  ylabel={\small $y$},
every axis x label/.style={
below,
at={(2.5cm,0cm)},
  yshift=-8pt
  },
every axis y label/.style={
below,
at={(0cm,1.25cm)},
  xshift=-8pt},
  y label style={rotate=90,anchor=south},
  xmin=-1.3,
  xmax=3.7,
  ymin=-2.6,
  ymax=9,  
  height=11cm,
] 
\addplot graphics[xmin=-1.401,xmax=3.8,ymin=-2.63,ymax=9] {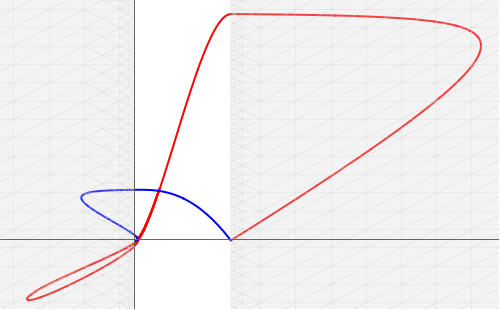};
\end{axis}
\end{tikzpicture} 
\caption{Shooting of $X_{\mathopen{[}0,1\mathclose{]}}$ forward over the interval $\mathopen{[}0,\tau\mathclose{]}$ (red) and shooting of $X_{\mathopen{[}0,1\mathclose{]}}$ backward over the interval $\mathopen{[}\tau,T\mathclose{]}$ (blue).
}
\end{subfigure}
\hspace*{\fill}
\begin{subfigure}[t]{.45\textwidth}
  \centering
\begin{tikzpicture}[scale=1]
\begin{axis}[
  x post scale=1.8,
  y post scale=0.25,
  tick label style={font=\scriptsize},
  enlargelimits=false, axis on top, axis equal image,
  xtick={0.24,0.25},
  ytick={1.95,2.00},
  xlabel={\small $x$},
  ylabel={\small $y$},
every axis x label/.style={
below,
at={(2cm,0cm)},
  yshift=-8pt
  },  
every axis y label/.style={
below,
at={(0cm,1.25cm)},
  xshift=-8pt},
  y label style={rotate=90,anchor=south},
  xmin=0.24,
  xmax=0.25,
  ymin=1.95,
  ymax=2,  
  height=11cm,
  scale only axis
] 
\addplot graphics[xmin=0.24,xmax=0.25,ymin=1.95,ymax=2.00] {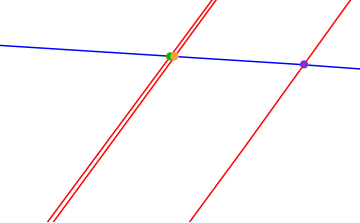};
\end{axis}
\end{tikzpicture} 
\caption{Zooming on three intersection points in $\Phi_{0}^{\tau}(X_{\mathopen{[}0,1\mathclose{]}})\cap\Phi_{T}^{\tau}(X_{\mathopen{[}0,1\mathclose{]}})$ which identify three solutions of $(\mathcal{N}_{\lambda,\mu})$.}
\end{subfigure} 
\caption{In the phase plane $(x,y)$: dynamics of the Poincar\'{e} maps $\Phi_{0}^{\tau}$ and $\Phi_{T}^{\tau}$ associated to system $(\mathcal{S}_{\lambda,\mu})$ as in the framework of Example with $\sigma=0.5$, $\tau=1$, $T=2$, $a_{1}=1.75$, $a_{2}=1$, $a_{3}=1$, for $\lambda=25$ and $\mu=500$.}        
\label{fig-2}
\end{figure}

We are now ready to prove Theorem~\ref{th-1.1}. Accordingly, we divide the argument in four steps. In the first three, we describe the main properties of the image of the segment $X_{\mathopen{[}0,1\mathclose{]}}$ through the Poincar\'{e} maps associated to the subintervals $\mathopen{[}0,\sigma\mathclose{]}$, $\mathopen{[}\sigma,\tau\mathclose{]}$ and $\mathopen{[}\tau,T\mathclose{]}$, respectively. Finally, in the last step we reach the thesis.

\medskip
\noindent
\textit{Step~1. Dynamics on $\mathopen{[}0,\sigma\mathclose{]}$. }
Let us fix $0<\kappa_{1}<\kappa_{0}<1$ and $0<t_{1} \leq \sigma (1-\kappa_{1}/\kappa_{0})$. In this manner, we have that $\kappa_{1}/(\sigma-t_{1}) \leq (\kappa_{0}-\kappa_{1})/t_{1}$ and so we can apply Lemma~\ref{lem-2.1} together with Lemma~\ref{lem-2.2}. Then, for $\lambda > \lambda^{\star}(\kappa_{0},\kappa_{1},t_{1})$
 (cf.~\eqref{eq-lambda1})
 and an arbitrary $\mu>0$,  we obtain that 
\begin{equation*}
x(\sigma;0,\kappa_{0},0)\leq0, \quad y(\sigma;0,\kappa_{0},0)<0.
\end{equation*}
We stress that this conclusion does not depend on $\mu$.
Next, we notice that $\Phi_{0}^{\sigma}(1,0)=(1,0)$ and, by the concavity of $x$ in $\mathopen{[}0,\sigma\mathclose{]}$, that $\Phi^{\sigma}_{0}(\mathopen{[}0,1\mathclose{]}\times\{0\})\subseteq\mathopen{]}-\infty,1\mathclose{]}\times\mathopen{]}-\infty,0\mathclose{]}$.
Thus, from the continuous dependence of the solutions upon the initial data and the Intermediate Value Theorem, the following fact holds. 
There exists an interval $\mathopen{[}l_{1},1\mathclose{]}\subseteq\mathopen{[}\kappa_{0},1\mathclose{]}$ such that $\Phi^{\sigma}_{0}(\mathopen{[}l_{1},1\mathclose{]}\times\{0\})\subseteq\mathopen{[}0,1\mathclose{]}\times\mathopen{]}-\infty,0\mathclose{]}$, $\Phi_{0}^{\sigma}(l_{1},0) \in \{0\}\times\mathopen{]}-\infty,0\mathclose{[}$ 
and 
$x(t;0,\xi,0)\in \mathopen{]}0,1\mathclose{[}$ for all $t\in \mathopen{[}0,\sigma\mathclose{]}$, $\xi \in \mathopen{]}l_{1},1\mathclose{[}$. 

Furthermore, by Lemma~\ref{lem-ang} there exits $\kappa_{4}\in\mathopen{]}0,\kappa_{1}\mathclose{[}$ such that $\Phi^{\sigma}_{0}(\mathopen{]}0,\kappa_{4}\mathclose{]}\times\{0\})\subseteq\mathopen{]}0,1\mathclose{[}\times\mathopen{]}-\infty,0\mathclose{]}$. Then, recalling that $\Phi_{0}^{\sigma}(\kappa_{0},0) \in \mathopen{]}-\infty,0\mathclose{]}\times\mathopen{]}-\infty,0\mathclose{[}$, from the same previous arguments of continuity, there exists an interval $\mathopen{[}0,r_{1}\mathclose{]}\subseteq\mathopen{[}0,\kappa_{0}\mathclose{]}$ (with $r_{1}>\kappa_{4}$) such that $\Phi^{\sigma}_{0}(\mathopen{[}0,r_{1}\mathclose{]}\times\{0\})\subseteq\mathopen{[}0,1\mathclose{[}\times\mathopen{]}-\infty,0\mathclose{]}$, $\Phi_{0}^{\sigma}(r_{1},0) \in \{0\}\times\mathopen{]}-\infty,0\mathclose{[}$
and 
$x(t;0,\xi,0)\in \mathopen{]}0,1\mathclose{[}$ for all $t\in \mathopen{[}0,\sigma\mathclose{]}$, $\xi \in \mathopen{]}0,r_{1}\mathclose{[}$.

\medskip
\noindent
\textit{Step~2. Dynamics on $\mathopen{[}\tau,T\mathclose{]}$. }
Analogously to \textit{Step~1}, let us fix $0<\kappa_{3}<\kappa_{T}<1$ and $0<t_{3} \leq \tau + (T-\tau)\kappa_{3}/\kappa_{T}$. 
Given $\lambda > \lambda^{\star\star}(\kappa_{3},\kappa_{T},t_{3})$ (cf.~\eqref{eq-lambda2})
and an arbitrary $\mu>0$, from Lemma~\ref{lem-2.4} and Lemma~\ref{lem-2.5} we have that
\begin{equation*}
x(\tau;T,\kappa_{T},0)\leq0, \quad y(\tau;T,\kappa_{T},0)>0.
\end{equation*}
Furthermore, we highlight that $\Phi_{T}^{\tau}(1,0)=(1,0)$ and $\Phi^{\tau}_{T}(\mathopen{[}0,1\mathclose{]}\times\{0\})\subseteq\mathopen{]}-\infty,1\mathclose{]}\times\mathopen{[}0,+\infty\mathclose{[}$.
Consequently, by the continuous dependence of the solutions upon the initial data and the Intermediate Value Theorem, there exists an interval $\mathopen{[}l_{2},1\mathclose{]}\subseteq\mathopen{[}\kappa_{T},1\mathclose{]}$ such that $\Phi^{\tau}_{T}(\mathopen{[}l_{2},1\mathclose{]}\times\{0\})\subseteq\mathopen{[}0,1\mathclose{]}\times\mathopen{[}0,+\infty\mathclose{[}$, $\Phi_{T}^{\tau}(l_{2},0) \in \{0\}\times\mathopen{]}0,+\infty\mathclose{[}$ 
and 
$x(t;T,\xi,0)\in \mathopen{]}0,1\mathclose{[}$ for all $t\in \mathopen{[}\tau,T\mathclose{]}$, $\xi \in \mathopen{]}l_{2},1\mathclose{[}$. 

\medskip
\noindent
\textit{Step~3. Dynamics on $\mathopen{[}\sigma,\tau\mathclose{]}$. }
Let us define
\begin{equation*}
\lambda^{*}:=\max\bigl{\{}\lambda^{\star}(\kappa_{0},\kappa_{1},t_{1}),\lambda^{\star\star}(\kappa_{3},\kappa_{T},t_{3})\bigr{\}}
\end{equation*}
and fix  $\lambda > \lambda^{*}$.

First of all, we observe that, for any $x_{0}\in\mathbb{R}$, the solution $(x(t),y(t))$ to system~$(\mathcal{S}_{\lambda,\mu})$ with initial values $x(0)=x_{0}$ and $y(0)=0$ satisfies
\begin{equation*}
y(\sigma) = y(0) + \lambda \int_{0}^{\sigma} a^{+}(\xi)g(x(\xi)) \,d\xi \geq - \omega_{\sigma},
\end{equation*}
where $\omega_{\sigma} := \lambda^{*} A^{+}(0,\sigma) \max_{s\in\mathopen{[}0,1\mathclose{]}}g(s)$.

Let us take $p_{1}\in\mathopen{]}0,r_{1}\mathclose{[}$ and $p_{2}\in\mathopen{]}l_{1},1\mathclose{[}$. We define 
\begin{equation*}
\kappa_{\sigma,i}:=x(\sigma;0,p_{i},0),\quad \text{for } \, i=1,2.
\end{equation*}
From the properties of the continua $\Phi_{0}^{\sigma}(\mathopen{[}0,r_{1}\mathclose{]}\times\{0\})$ and $\Phi_{0}^{\sigma}(\mathopen{[}l_{1},1\mathclose{]}\times\{0\})$ achieved in \textit{Step~1}, it follows that $\kappa_{\sigma,i}\in\mathopen{]}0,1\mathclose{[}$ for $i=1,2$.
Next, for $i=1,2$, we fix $\kappa_{2,i}\in\mathopen{]}\kappa_{\sigma,i},1\mathclose{[}$ and choose $t_{2,i}$ such that
\begin{equation*}
\sigma< t_{2,i} \leq \min\biggl{\{}\sigma + \dfrac{\kappa_{\sigma,i}}{2\omega_{\sigma}},
\dfrac{\sigma(1-\kappa_{2,i})+\tau(\kappa_{2,i}-\kappa_{\sigma,i})}{1-\kappa_{\sigma,i}} 
\biggr{\}}.
\end{equation*}
In this manner, we enter in the setting of Lemma~\ref{lem-2.6} and Lemma~\ref{lem-2.7}.
For $i=1,2$, taking $\mu>\mu^{\star}(\kappa_{2,i},\kappa_{\sigma,i},t_{2,i},\omega_{\sigma})$ (cf.~\eqref{eq-mu}), we obtain that 
\begin{equation}\label{eq-p}
x(\tau;0,p_{i},0)\geq1,\quad y(\tau;0,p_{i},0)>\omega_{i}>0,\quad\text{for } \, i=1,2.
\end{equation}

We remark now that, for any choice of $t_{0}\in\mathopen{[}0,T\mathclose{]}$ and $y_{0}<0$, if $(x(t),y(t))$ is the solution of the Cauchy problem associated with system~$(\mathcal{S}_{\lambda,\mu})$ satisfying the initial conditions $x(t_{0})=0$ and $y(t_{0})=y_{0}$, then 
\begin{equation*}
x(t;t_{0},0,y_{0})<0,\quad y(t;t_{0},0,y_{0})<0, \quad \text{for all } \, t\in \mathopen{]}t_{0},T\mathclose{]}.
\end{equation*}
Indeed, let $\mathopen{]}t_{0},t^{*}\mathclose{[}\in\mathopen{]}t_{0},T\mathclose{]}$ be the maximal open interval such that 
$y(t)<0$ for all $t\in \mathopen{]}t_{0},t^{*}\mathclose{[}$. By an integration of $x'=y$, we have $x(t)<0$ for all $t\in\mathopen{]}t_{0},t^{*}\mathclose{[}$. Assume now, by contradiction, that $t^{*}<T$. Then, $0=y(t^{*})=y_{0}<0$ and we have a contradiction. The claim follows. 

Consequently, we deduce that 
\begin{equation}\label{eq-r1}
x(\tau;0,r_{1},0)<0,\quad y(\tau;0,r_{1},0)<0,
\end{equation}
and 
\begin{equation}\label{eq-r2}
x(\tau;0,l_{1},0)<0,\quad y(\tau;0,l_{1},0)<0.
\end{equation}

At this point, taking into account \eqref{eq-p}, \eqref{eq-r1}, \eqref{eq-r2} and $\Phi_{0}^{\tau}(0,0)=(0,0)$, thanks to
the continuous dependence of the solutions upon the initial data and the Intermediate Value Theorem, we deduce what follows.
There exist three intervals 
\begin{equation*}
\mathopen{[}q_{1,1},q_{2,1}\mathclose{]}\subseteq\mathopen{[}0,p_{1}\mathclose{]}, \quad
\mathopen{[}q_{1,2},q_{2,2}\mathclose{]}\subseteq\mathopen{[}p_{1},r_{1}\mathclose{]}, \quad
\mathopen{[}q_{1,3},q_{2,3}\mathclose{]}\subseteq\mathopen{[}l_{1},p_{2}\mathclose{]},
\end{equation*}
such that, for each $j\in\{1,2,3\}$,
$\Phi_{0}^{\tau}(\mathopen{[}q_{1,j},q_{2,j}\mathclose{]}\times\{0\})\subseteq\mathopen{[}0,1\mathclose{]}\times\mathbb{R}$ with 
\begin{equation*}
\Phi_{0}^{\tau}(q_{1,j},0) \in \{0\}\times\mathopen{]}-\infty,0\mathclose{]},
\quad
\Phi_{0}^{\tau}(q_{2,j},0) \in \{1\}\times\mathopen{]}0,+\infty\mathclose{[},
\end{equation*}
and
\begin{equation*}
x(t;0,\xi,0)\in \mathopen{]}0,1\mathclose{[}, \quad \text{for all } \, t\in \mathopen{[}0,\tau\mathclose{]}, \; \xi \in \mathopen{]}q_{1,j},q_{2,j}\mathclose{[}. 
\end{equation*}

We conclude that there exist three sub-continua of $\Phi_{0}^{\tau}(X_{\mathopen{[}0,1\mathclose{]}})$ connecting $\{0\}\times\mathopen{]}-\infty,0\mathclose{]}$ with $\{1\}\times\mathopen{]}0,+\infty\mathclose{[}$.
We stress that the three sub-continua do not intersect each other, due to the uniqueness of the solutions to the initial value problems associated with $(\mathcal{S}_{\lambda,\mu})$.

\medskip
\noindent
\textit{Step~4. Conclusion. }
Let us take 
\begin{equation*}
\mu > \mu^{*}(\lambda):=\max_{i\in\{1,2\}}\mu^{\star}(\kappa_{2,i},\kappa_{\sigma,i},t_{2,i},\omega_{\sigma}).
\end{equation*}
Then, we are in the following situation.
\begin{itemize}
\item From \textit{Step~2}, we deduce the existence of a sub-continuum in $\Phi_{T}^{\tau}(X_{\mathopen{[}0,1\mathclose{]}})$ connecting $\{0\}\times\mathopen{]}0,+\infty\mathclose{[}$ with $(1,0)$.
\item From \textit{Step~1} and \textit{Step~3}, we deduce the existence of three pairwise disjoint sub-continua in $\Phi_{0}^{\tau}(X_{\mathopen{[}0,1\mathclose{]}})$ connecting $\{0\}\times\mathopen{]}-\infty,0\mathclose{]}$ with $\{1\}\times\mathopen{]}0,+\infty\mathclose{[}$.
\end{itemize}
This way, from a standard connectivity argument, it follows the existence of three distinct intersection points:
\begin{equation*}
P_{j}\in\Phi_{0}^{\tau}(\mathopen{]}q_{1,j},q_{2,j}\mathclose{[}\times\{0\})\cap\Phi_{T}^{\tau}(\mathopen{]}l_{2},1\mathclose{[}\times\{0\}), \quad j=1,2,3.
\end{equation*}
See Figure~\ref{fig-2} for a graphical representation.
For each $j\in\{1,2,3\}$, given the solution $(x(t),y(t))$ of the Cauchy problem associated with system~$(\mathcal{S}_{\lambda,\mu})$ with initial data at time $t=\tau$ the point $P_{j}$, then we have a positive solution to problem~$(\mathcal{N}_{\lambda,\mu})$ defined by $u(t):=x(t;\tau,P_{j})$. 
Moreover, from a straightforward argument by contradiction, it follows that
\begin{equation*}
\begin{aligned}
&\Phi_{0}^{t}(\xi,0) \in \mathopen{]}0,1\mathclose{[}\times\mathbb{R}, &&\text{for all } \, t\in\mathopen{]}q_{1,j},q_{2,j}\mathclose{[}, \; \xi\in\mathopen{[}0,\tau\mathclose{]},
\\&\Phi_{T}^{t}(\xi,0) \in \mathopen{]}0,1\mathclose{[}\times\mathbb{R}, &&\text{for all } \, t\in\mathopen{]}l_{1},1\mathclose{[}, \; \xi\in\mathopen{[}\tau,T\mathclose{]},
\end{aligned}
\end{equation*}
and so we have that $0<u(t)<1$ for all $t\in\mathopen{[}0,T\mathclose{]}$. Then, Theorem~\ref{th-1.1} is proved.\qed

\begin{remark}\label{rem-2.2}
It is worth noting that the hypotheses on $g(s)$ are not all used in the study of the dynamics over the three intervals $\mathopen{[}0,\sigma\mathclose{]}$, $\mathopen{[}\sigma,\tau\mathclose{]}$ and $\mathopen{[}\tau,T\mathclose{]}$. In particular, condition $(g_{0})$ is not used in Section~\ref{section-2.2} and Section~\ref{section-2.3}. Accordingly, Theorem~\ref{th-1.1} can be stated in the following more general form. Assume that $g_{1},g_{2},g_{3} \colon \mathopen{[}0,1\mathclose{]} \to \mathbb{R}^{+}$ are locally Lipschitz continuous functions satisfying $(g_{*})$. Suppose that $g_{1}(s)$ satisfies condition $(g_{0})$. Given three non-null weights $a_{1}\in L^{1}(\mathopen{[}0,\sigma\mathclose{]},\mathopen{[}0,+\infty\mathclose{[})$, $a_{2}\in L^{1}(\mathopen{[}\sigma,\tau\mathclose{]},\mathopen{[}0,+\infty\mathclose{[})$ and $a_{3}\in L^{1}(\mathopen{[}\tau,T\mathclose{]},\mathopen{[}0,+\infty\mathclose{[})$ with $0<\sigma<\tau<T$, then there exists $\lambda^{*}>0$ such that for each $\lambda>\lambda^{*}$ there exists $\mu^{*}(\lambda)>0$ such that for every $\mu>\mu^{*}(\lambda)$ the Neumann problem associated with the following differential equation
\begin{equation}\label{eq-rem}
u'' + \lambda\bigl{(} a_{1}(x)g_{1}(x) + a_{3}(x)g_{3}(x)\bigr{)} - \mu a_{2}(x)g_{2}(x) = 0
\end{equation}
has at least three positive solutions.

In particular, we observe that our result is valid also for functions $g_{2}, g_{3}$ such that $g_{2}(s)/s \not\to 0$ and $g_{3}(s)/s \not\to 0$ as $s\to0^{+}$, as in the case of the map $s\mapsto s(1-s)$.

An equation of the form~\eqref{eq-rem} involves a \textit{conflicting nonlinearity}. The adjective ``conflicting'' refers to the fact that the term in the nonlinearity depending on $\lambda$ has an opposite effect on the existence of solutions with respect to the one depending on $\mu$. Such kind of problems are apparently new in this framework but  have already  been addressed in the context of superlinear nonlinearities (see \cite{Fe-17cpaa,Ru-98} for an introduction and references on this topic). 
$\hfill\lhd$
\end{remark}

\section{Related results}\label{section-3}

Dealing with indefinite Neumann problems of the form considered in $(\mathcal{N}_{\lambda,\mu})$, classical applications are both in the context of periodic boundary value problems and in the one of radially symmetric Neumann boundary value problems defined on an annular domain of $\mathbb{R}^{N}$ for $N\geq2$.

\medskip

Regarding the first context, given a solution of the Neumann problem $(\mathcal{N}_{\lambda,\mu})$ one can easily prove the following result by means of an even reflection and a periodic extension. 

\begin{corollary}\label{cor-1}
Let $g \colon \mathopen{[}0,1\mathclose{]} \to \mathbb{R}^{+}$ be a locally Lipschitz continuous function satisfying $(g_{*})$ and $(g_{0})$.
Let $a \in L^{1}_{\text{\rm loc}}(\mathbb{R})$ be a $2T$-periodic even function satisfying $(a_{*})$.
Then, there exists $\lambda^{*}>0$ such that for each $\lambda>\lambda^{*}$ there exists $\mu^{*}(\lambda)>0$ such that for every $\mu>\mu^{*}(\lambda)$ the equation
$u'' + (\lambda a^{+}(t)-\mu a^{-}(t)) g(u) = 0$
has at least three $2T$-periodic positive solutions.
\end{corollary}

\medskip

On the other context, in $\mathbb{R}^{N}$ (for $N\geq2$) let us consider the open annular domain
\begin{equation*}
\Omega := \bigl{\{} x\in\mathbb{R}^{N} \colon R_{i} < |x| < R_{e} \bigr{\}},\quad \text{with } \, 0<R_{i}<R_{e},
\end{equation*}
where $|\cdot|$ denotes the usual Euclidean norm in $\mathbb{R}^{N}$.
We deal with the indefinite Neumann problem 
\begin{equation*}
\begin{cases}
\, - \Delta u = w_{\lambda,\mu}(x)g(u)  & \text{in } \Omega, \\\vspace*{2pt}
\, \dfrac{\partial u}{\partial \nu} = 0 & \text{on } \partial\Omega,
\end{cases}
\leqno{(\mathcal{P}_{\lambda,\mu})}
\end{equation*}
where 
\begin{equation*}
w_{\lambda,\mu}(x):=\lambda w^{+}(x)-\mu w^{-}(x).
\end{equation*}
Assuming that the weight term has radially symmetry, namely $w(x)=\mathcal{W}(|x|)$ for all $x\in\Omega$ with $\mathcal{W}\colon \mathopen{[}R_{i},R_{e}\mathclose{]}\to\mathbb{R}$,
we look for radially symmetric positive solutions to problem $(\mathcal{P}_{\lambda,\mu})$, i.e.~solutions of the form $u(x)=\mathcal{U}(|x|)$ where $\mathcal{U}\colon \mathopen{[}R_{i},R_{e}\mathclose{]}\to\mathbb{R}$.
Accordingly, our study can be reduced to the search of positive solutions of the Neumann boundary value problem
\begin{equation}\label{eq-rad}
{\mathcal{U}}''(r) + \dfrac{N-1}{r} \, {\mathcal{U}}'(r) + {\mathcal{W}}_{\lambda,\mu}(r) g({\mathcal{U}}(r)) = 0,
\quad {\mathcal{U}}'(R_{i}) = {\mathcal{U}}'(R_{e}) = 0,
\end{equation}
which, via a standard change of variable, can be transformed into the equivalent problem
\begin{equation}\label{eq-rad1}
v'' + a_{\lambda,\mu}(t) g(v) = 0, \quad v'(0) = v'(T) = 0,
\end{equation}
with
\begin{equation*}
t = h(r):= \int_{R_{i}}^{r} \xi^{1-N} \,d\xi, \quad r(t):= h^{-1}(t),
\end{equation*}
and
\begin{equation*}
v(t):={\mathcal{U}}(r(t)), \quad a(t):= r(t)^{2(N-1)}{\mathcal{W}}(r(t)), \quad T:= \int_{R_{i}}^{R_{e}} \xi^{1-N}\,d\xi.
\end{equation*}

In this setting, a direct consequence of Theorem~\ref{th-1.1} is the following.

\begin{corollary}\label{cor-2}
Let $g \colon \mathopen{[}0,1\mathclose{]} \to \mathbb{R}^{+}$ be a locally Lipschitz continuous function satisfying $(g_{*})$ and $(g_{0})$.
Let $\mathcal{W} \colon \mathopen{[}R_{i},R_{e}\mathclose{]} \to \mathbb{R}$ be an $L^{1}$-function such that there exist 
$\sigma,\tau$ with $0 < \sigma < \tau < T$ for which the following holds
\begin{equation*}
\begin{aligned}
&\mathcal{W}^{+}(t)\succ 0,  & &\mathcal{W}^{-}(t)\equiv 0, & &\text{on } \mathopen{[}R_{i},\sigma \mathclose{]},\\
&\mathcal{W}^{+}(t)\equiv 0, & &\mathcal{W}^{-}(t)\succ 0,  & &\text{on } \mathopen{[}\sigma,\tau \mathclose{]}, \\
&\mathcal{W}^{+}(t)\succ 0,  & &\mathcal{W}^{-}(t)\equiv 0, & &\text{on } \mathopen{[}\tau,R_{e} \mathclose{]},
\end{aligned}
\end{equation*}
and let $w(x) := \mathcal{W}(|x|)$, for $x\in\Omega$.
Then, there exists $\lambda^{*}>0$ such that for each $\lambda>\lambda^{*}$ there exists $\mu^{*}(\lambda)>0$ such that for every $\mu>\mu^{*}(\lambda)$ problem $(\mathcal{P}_{\lambda,\mu})$ has at least three radially symmetric positive solutions.
\end{corollary}

\section*{Acknowledgements}

We thank Prof.~Fabio Zanolin for the interesting and helpful discussions on the subject of the present paper.

\bibliographystyle{elsart-num-sort}
\bibliography{FeSo_biblio}

\bigskip
\begin{flushleft}

{\small{\it Preprint}}

{\small{\it June 2017}}

\end{flushleft}

\end{document}